\documentclass[12pt,a4paper,reqno]{amsart}
\usepackage{amsthm,amsfonts,amsmath,amssymb}
\usepackage{mathrsfs}
\usepackage[top=1.6in, bottom=1.3in, left=1.3in, right=1.3in]{geometry}
\usepackage[colorlinks,linkcolor=blue,citecolor=blue,urlcolor=blue,hypertexnames=true]{hyperref}
\usepackage{enumerate}

\author{Ayan Nath}
\address{Kaliabor College, Kuwaritol, Assam, India}
\email{ayannath7744@gmail.com}

\author{Abhishek Jha}
\address{Indraprastha Institute of Information Technology, New Delhi, India}
\email{abhishek20553@iiitd.ac.in}

\subjclass[2010]{Primary: 11N32, Secondary: 11A41}

\keywords{Polynomial; Primes; Least Common Multiple; Greatest Prime Divisor.}

\title[On the L.C.M. of Polynomial Sequences at Prime Arguments]
{On the Least Common Multiple of Polynomial Sequences at Prime Arguments}

\theoremstyle{theorem}
\newtheorem{theorem}{Theorem}[section]
\newtheorem{lemma}[theorem]{Lemma}
\newtheorem{corollary}[theorem]{Corollary}
\newtheorem{question}[theorem]{Question}

\newtheorem{proposition}[theorem]{Proposition}
\newtheorem{remark}[theorem]{Remark}

\newcommand{\li}{\operatorname{Li}}
\newcommand{\e}{\varepsilon}
\newcommand{\lcm}{\operatorname{lcm}}
\newcommand{\disc}{\operatorname{disc}}
\newcommand{\Mod}[1]{\ (\mathrm{mod}\ #1)}
\newcommand{\rad}{\operatorname{rad}}
\newcommand{\logx}{\log x}

\newcommand{\von}{\Lambda}
\newcommand{\ZZ}{\mathbb Z}

\newcommand{\xx}{x_{\mathfrak b}}
\newcommand{\xdelta}{x^{\delta}}
\newcommand{\doi}[1]{doi:\,\href{http://dx.doi.org/#1}{#1}}
\renewcommand{\O}{\mathcal O}

\begin{document}

\maketitle

\begin{abstract}
    Cilleruelo conjectured that if $f\in\ZZ[x]$ is an irreducible polynomial of degree $d\ge 2$ then, $\log \lcm \{f(n)\mid n<x\} \sim (d-1)x\log x.$ In this article, we investigate the analogue of prime arguments, namely, $\lcm \{f(p)\mid p<x\},$ where $p$ denotes a prime and obtain non-trivial lower bounds on it. Further, we also show some results regarding the greatest prime divisor of $f(p).$
\end{abstract}

\section{Introduction}
For a polynomial $f\in \ZZ[x],$ define $L_f(x)=\lcm\{f(n)\mid n < x \text{ and }f(n)\ne 0\},$ where the lcm of an empty set is taken to be $1.$ The Prime Number Theorem is equivalent to 
$$\log \lcm\{1,2,\ldots,n\}\sim n.$$
Therefore, we expect similar rate of growth for the case when $f$ is a product of linear polynomials; see the article by Hong, Qian, and Tan \cite{hong} for a thorough analysis of this case.
However, the growth is not the same for higher degree polynomials.
Cilleruelo in \cite{cill} conjectured that $\log L_f(x) \sim (d-1)x\log x$ for irreducible polynomials $f$ of degree $d\ge 2$ and proved it for $d=2$. For some time, $\log L_f(x) \gg x$ proven by Hong, Luo, Qian, and Wang in \cite{hong:lower-bound}, for polynomials with non-negative integer coefficients, was the strongest bound known.
Recently, the conjectured order of growth was obtained by Maynard and Rudnick in \cite{maynard} and the bound was improved to $x\log x$ by Sah in \cite{sah}. For a thorough survey on the least common multiple of polynomial sequences, see \cite{sanna}.

In this article, we study the analogous problem at prime arguments.
From the Prime Number Theorem, we know that
$$\log \lcm\{p\mid p < x\} \sim x.$$
This motivates us to consider 
$\lcm\{f(p) \mid p < x\}$ for an arbitrary polynomial $f \in \ZZ[x].$ For simplicity, we will only consider irreducible polynomials $f.$

\begin{theorem}\label{main-theorem}
    Let $f\in\ZZ[x]$ be an irreducible polynomial of degree $d$. Then,
    $$\log \lcm \{f(p)\mid p<x\}\gg x^{1-\e(d)},$$
    where $\e(1)=0.3735,~ \e(2)=0.153$ and $\e(d)=\exp\left(\frac{-d-0.9788}{2}\right)$ for $d\ge 3.$
\end{theorem}
We remark that $\log \lcm\{f(p)\mid p<x\}\le (d+o(1))x\ll x$ follows from the Prime Number Theorem.  

There is a lot of literature on the subject of largest prime divisor of $p+a$ for some fixed integer $a.$ Goldfeld in \cite{goldfeld} showed that there is a positive proportion of primes $p$ such that $p+a$ has a prime divisor greater than $p^{\delta}$ for $\delta=0.5.$ The strongest known result in this regard is $\delta=0.677$ proven by Baker and Harman in \cite[Theorem 8.3]{prime-detecting-sieves}, an improvement of $\delta=0.6687$ obtained by Fouvry in \cite{fouvry}. Luca in \cite{luca:p} obtained lower bounds on the proportion of such primes $p$ for $\delta \in[\frac 14,\frac 12].$ Similar work is also done for quadratic polynomials. Wu and Xi in \cite{uwu} proved that there exist infinitely many primes $p$ such that $p^2+1$ has a prime divisor greater than $p^{0.847}$ by virtue of the Quadratic Brun-Titchmarsh theorem (see Theorem \ref{qbrun}) developed by the authors.

We obtain a result of a similar flavor for general polynomials which we state as follows.
\begin{theorem}\label{density-result}
    Let $f\in\ZZ[x]$ be an irreducible polynomial of degree $d.$
    Then, there is a positive proportion of primes $p$ such that $f(p)$ has a prime divisor greater than $p^{1-\varepsilon(d)},$ where $\e(1)=0.3735,~ \e(2)=0.153$ and $\e(d)=\exp\left(\frac{-d-0.9788}{2}\right)$ for $d\ge 3.$
\end{theorem}

The following table shows some values of $1-\varepsilon(d)$ for various $d.$

\begin{table}[ht]
\caption{Values of $1-\varepsilon(d)$}
\begin{tabular}{ |c|c|c|c|c|c|c|c|c| } 
 \hline
 $d$                & 1 & 2 & 3 & 4 & 5 & 6 & 7 & 8 \\ \hline
 $1-\varepsilon(d)$ & 0.6265 & 0.847 & 0.8632 & 0.9170 & 0.9496 & 0.9694 & 0.9814 & 0.9887  \\
 \hline
\end{tabular}
\label{table}
\end{table}

\noindent \textbf{Notations.} We employ Landau-Bachmann notations $\O$ and $o$ as well as their associated Vinogradov notations $\ll$ and $\gg.$ We say that $a(x)\sim b(x)$ if $$\lim_{x\to \infty} \frac{a(x)}{b(x)}=1.$$ As usual, define $\pi(x;m,a)$ to be the number of primes $p<x$ such that $p\equiv a\Mod m.$ Throughout the article, $p$ and $q$ will denote primes, and we fix an irreducible polynomial $f \in \ZZ[x]$ of degree $d\ge 1.$ We will often suppress the dependence of constants on $f$. At places, we may use Mertens' first theorem without commentary.

\section{Background}
\begin{theorem}[Brun-Titchmarsh, \cite{brun}]\label{theta-bound-vanilla}
    Let $\theta=\frac{\log m}{\log x},$ where $\theta\in(0, 1).$ Then,
    $$\pi(x;m,a)<(C(\theta)+o(1))\cdot\frac{x}{\phi(m)\log x}$$
    where 
    $$C(\theta)=
        \frac{2}{1-\theta}.$$
\end{theorem}

\begin{corollary}\label{brun-weak}
    Let $\e>0$ be a constant. Then,
    $$\pi(x;m,a)\ll_{\e} \frac{x}{\phi(m)\logx}$$
    for all positive integers $m<x^{1-\e}.$
\end{corollary}

\begin{theorem}[Iwaniec, \cite{iwaniec}] \label{theta-bound-iwaniec}
    Let $\theta=\frac{\log m}{\log x}$ where $\theta\in[\frac 9{10}, \frac 23].$ Then,
    $$\pi(x;m,a)<(C(\theta)+o(1))\cdot\frac{x}{\phi(m)\log x},$$
    where 
    $$C(\theta)=
        \frac{8}{6-7\theta}.$$
\end{theorem}

\begin{theorem}[Wu and Xi, \cite{qbrun}]\label{qbrun}
    Let $A>0$ and $f(x)$ be an irreducible quadratic polynomial. Define $\varsigma(m)=\#\{p<x\mid f(p)\equiv 0 \Mod m\}$ and $\rho(m)$ to be the number of solutions of the congruence $f(x)\equiv 0 \Mod m.$ For large $L=x^{\theta}$ with $\theta\in[\tfrac 12,\tfrac{16}{17}),$
    we have
    $$\varsigma(m)\le(C(\theta)+o(1))\rho(m)\cdot \frac{x}{\phi(m)\log x},$$
    for all $m\in[L,2L]$ with at most $\O_A(L/(\log L)^A)$ exceptions,
    where 
    $$C(\theta)=\begin{cases}
        \frac{124}{91-89\theta} &, \text{ if }\theta\in[\frac 12,\frac{64}{97})\\
        \frac{120}{86-83\theta} &, \text{ if }\theta\in[\frac{64}{97},\frac{32}{41})\\
        \frac{28}{19-18\theta} &, \text{ if }\theta\in[\frac{32}{41},\frac{16}{17}).
    \end{cases}$$
\end{theorem}

\begin{theorem}[Bombieri-Vinogradov]\label{bomb}
    Let $A\ge 6$ and $Q\le x^{\tfrac 12}/(\log x)^A.$ Then,
    $$\sum_{q\le Q}\max_{2\le y\le x} \max_{(a,q)=1}\left|\pi(y;q,a)-\frac{y}{\phi(q)\log y}\right|\ll_A \frac{x}{(\log x)^{B}},$$
    where $B=A-5.$ 
\end{theorem}

\begin{lemma}\label{merten}
    Let $f$ be an irreducible integer polynomial and $\rho(m)$ be the number of roots of the congruence $f(x)\equiv 0 \Mod m.$ Then,
    $$\sum_{p<x} \frac{\rho(p)\log p}{p-1}=\log x + R +o(1)$$
    for some constant $R.$
\end{lemma}

\begin{proof}
    By \cite[3.3.3.5]{serre}, we have that 
    $$\sum_{p<x}\rho(p)=\li(x) + \O\left(\frac{x}{(\log x)^3}\right),$$
    where $\li(x)$ is the logarithmic integral. Applying Abel summation formula,
    \begin{align*}
        \sum_{p<x} \frac{\rho(p)\log p}{p} &= \frac{\log x}{x}\sum_{p<x}\rho(p)+\int_{2}^x\frac{\log x-1}{x^2}\left(\sum_{p<u}\rho(p)\right)\mathrm{d}u+C_0\\
                                           &= C_0+1+\O\left(\frac{1}{\log x}\right) + \int_{2}^x\frac{\log u-1}{u^2}\li(u)\mathrm{d}u+\O\left(\int_{2}^x\frac{\log u-1}{u(\log u)^3}\mathrm{d}u\right)\\
                                           &= \log x + C_1 + \O\left(\frac{1}{\log x}\right)
    \end{align*}
    for some constants $C_0$ and $C_1$.
    And the sum
    $$\sum_{p<x}\frac{\rho(p)\log p}{p-1}-\sum_{p<x}\frac{\rho(p)\log p}{p}=\sum_{p<x}\frac{\rho(p)\log p}{p(p-1)}$$
    is $C_2+o(1)$ for some constant $C_2$. Hence, our lemma is proved.
\end{proof}

\section{Proof of Theorem \ref{main-theorem}}

\subsection{Setup}
We study the product defined by 
$$Q(x)=\prod_{q<x}|f(q)|=\prod_{p}p^{\alpha_p(x)}$$
and exploit the fact that the contribution of prime factors less than $\xdelta$ is negligible compared to that of prime factors greater than $\xdelta,$ where $\delta$ is a parameter in $(\tfrac 12,1)$ to be chosen later. For some large enough constant $B,$ set $\xx=x^{1/2}(\log x)^{-B}$ for brevity. 

Define $\varrho(m)$ to be the set of residues modulo $m$ which satisfy the congruence $f(x)\equiv 0 \Mod m$ and $\rho(m)$ to be the cardinality of $\varrho(m).$ Note that we have $\rho(m)\le d$ by Lagrange's theorem and that if $p\nmid \disc f$ then $\rho(p)=\rho(p^n)$ for all $n\ge 2$ by Hensel's lemma.
    Also define $\varsigma(m)$ to be the sum
    $$\sum_{r\in\varrho(m)}\pi(x;m,r),$$
    the number of elements in $\{f(p)\mid p<x\}$ divisible by $m.$

\subsection{Estimating small primes}
We define 
$$Q_S(x) = \prod_{p<\xx} p^{\alpha_p(x)},$$
the part of $Q(x)$ consisting of small prime divisors. The main result here is the following.

\begin{proposition}\label{small-prime-estimate}
    $\log Q_S(x) = \frac x2 - \frac{Bx\log \log x}{\log x} +\O\left(\frac{x}{\log x}\right)$
\end{proposition}

The proof uses an estimate on $\alpha_p(x)$ making it easy to directly apply the Bombieri-Vinogradov theorem (Theorem \ref{bomb}) in the end. The following result is proved by standard analysis involving Hensel's lemma and the Brun-Titchmarsh theorem (Corollary \ref{brun-weak}).
\begin{lemma}\label{alpha-estimate}
    Let $p$ be a prime. If $p\nmid \disc f,$ then
    $$\alpha_p(x)=\sum_{p^n<\xx}\varsigma(p^n)+\O\left(\frac{x}{\max\{p,\xx\}\log x} + \frac{(\log x)^{2B}}{\log p}\right);$$
    else if $p\mid \disc f,$ we have 
    $$\alpha_p(x) = \varsigma(p).$$
\end{lemma}

\begin{proof}
    The case when $p\mid\disc f$ is easy to solve. So, let us assume $p\nmid \disc f.$  
    Observe that
    $$\alpha_p(x)=\sum_{n=1}^{\infty}\varsigma(p^n).$$
    When $p^n\ge x,$ we see that $\varsigma(p^n)\le \rho(p^n)\le d.$ If $p^n$ divides $f(k)$ for some $1\le k \le x,$ we have $p^n\le f(k)\le f(x)<x^{d+1}$, which implies that $n<(d+1)\frac{\log x}{\log p}.$
    Thus,
    $$\alpha_p(x)=\sum_{n=1}^{\infty}\varsigma(p^n)=\sum_{p^n<x}\varsigma(p^n)+\O\left(\frac{\log x}{\log p}\right).$$
    We split the summation into three intervals: $p^n\in[1,\xx]\cup(\xx,x^{0.9}]\cup (x^{0.9},x).$
    The last summation is 
    \begin{align*}
        \sum_{p^n\in(x^{0.9},x)}\varsigma(p^n)\le \sum_{p^n\in(x^{0.9},x)}\sum_{r\in\varrho(p^n)}\left(\frac{x}{p^n}+1\right)\le \sum_{p^n\in(x^{0.9},x)}\rho(p^n)(x^{0.1}+1)\ll x^{0.2}.
    \end{align*}
    By Corollary \ref{brun-weak}, the second summation is
    \begin{align*}
        \sum_{p^n\in(\xx, x^{0.9}]}\varsigma(p^n) &\ll \frac{\rho(p)x}{\log x}\sum_{\xx <p^n\le x^{0.9}} \frac{1}{\phi(p^n)}\\
                                                  &\ll \frac{x}{\max\{p,\xx\}\log x}+ \frac{x}{\log x}\sum_{\substack{n\ge 2 \\ \xx <p^n\le x^{0.9}}} \frac{1}{p^n}\\
                                                                          &\ll \frac{x}{\max\{p,\xx\}\log x}+\frac{x}{\log x}\cdot\frac{\log x}{\log p}\cdot\frac{1}{p^2}\\
                                                                          &\ll \frac{x}{\max\{p,\xx\}\log x} + \frac{(\log x)^{2B}}{\log p}.
    \end{align*}
    Thus, our lemma is proved.
\end{proof}

\begin{proof}[Proof of Proposition \ref{small-prime-estimate}]
    Using Lemma \ref{alpha-estimate},
    \begin{align*}
        \log Q_S(x) &= \sum_{p<\xx}\alpha_p(x) \log{p}\\
                    &= \sum_{p<\xx}\left(\sum_{p^n<\xx}\varsigma(p^n)+\O\left(\frac{x}{\xx\log x} + \frac{(\log x)^{2B}}{\log p}\right)\right) \log{p}\\
                    &= \sum_{m<\xx}\varsigma(m)\von(m) + \O\left(\frac{x}{\log x}\right).
    \end{align*}
    Using Theorem \ref{bomb} and Lemma \ref{merten}, we can estimate the above sum as
    \begin{align*}
        \sum_{m<\xx}\varsigma(m)\von(m) &= \frac{x}{\log x}\sum_{m<\xx}\frac{\rho(m)\von(m)}{\phi(m)}+\O\left(\frac{x}{(\log x)^{B-5}}\right)\\
                                                            &= \frac{x}{\log x}(\tfrac 12\logx - B\log\log x)+\O\left(\frac{x}{(\log x)^{B-5}}\right)\\
                                                            &= \frac x2 - \frac{Bx\log \log x}{\log x} +\O\left(\frac{x}{\log x}\right),
    \end{align*}
   proving the result. 
\end{proof}

\subsection{Removing medium-sized primes}
Define the product $$Q_M(x)=\prod_{\xx\le p\le x^{1/2}}p^{\alpha_p(x)},$$
the part of $Q(x)$ consisting of medium-sized primes. The main result of this section is the following. 
\begin{proposition}\label{medium-prime-estimate}
    $\log Q_M(x)\ll \frac{x\log\log x}{\log x}.$
\end{proposition}
This means we can just remove medium-sized primes from $\log Q(x)$ and only lose a sublinear portion.
The proof is a simple computation using Lemma \ref{alpha-estimate}.
\begin{proof}[Proof of Proposition \ref{medium-prime-estimate}]
    From Lemma \ref{alpha-estimate}, it follows that
    \begin{align*}
        \log Q_M(x) &= \sum_{\xx\le p \le x^{1/2}} \alpha_p(x) \log p\\ 
                    &\ll \sum_{\xx\le p \le x^{1/2}} \left(\frac{x}{p\log x} + \frac{(\log x)^{2B}}{\log p}\right) \log p\\
                    &= \frac{x}{\log x}\sum_{\xx\le p \le x^{1/2}} \frac{\log p}{p} + \O(x^{1/2}(\log x)^{2B})\\
                    &\ll \frac{x\log\log x}{\log x},
    \end{align*}
    as desired.
\end{proof}

\subsection{Bounding large primes}
Define the product $$Q_L(x)=\prod_{x^{1/2}<p<\xdelta}p^{\alpha_p(x)},$$
the part of $Q(x)$ consisting of large primes. The main result of this section is the following. 
\begin{proposition}\label{large-prime-estimate}
    $\log Q_L(x) \le (1+o(1))x \int_{1/2}^{\delta} C(\theta) ~\mathrm{d}\theta$.
\end{proposition}
The proof uses the Brun-Titchmarsh theorem (Theorem \ref{theta-bound-vanilla} and \ref{theta-bound-iwaniec}) and involves standard procedures to convert sums over primes to integrals.
\begin{proof}[Proof of Proposition \ref{large-prime-estimate}]
    Let $p$ be a prime in $(x^{1/2},\xdelta).$ Similar to the proof of Lemma \ref{alpha-estimate}, we have
    \begin{align*}
        \alpha_p(x) &= \sum_{n=1}^{\infty}\varsigma(p^n)= \varsigma(p) + \O(\log x/\log p)= \varsigma(p) + \O(1)
    \end{align*}
    as $p^2>x.$
    Therefore,
    \begin{align*}
        \log Q_L(x) &= \sum_{x^{1/2}<p<\xdelta} \alpha_p(x)\log p\\
                    &= \sum_{x^{1/2}<p<\xdelta} \varsigma(p)\log p + O(\xdelta).
    \end{align*}
    By Theorem \ref{theta-bound-vanilla}, \ref{theta-bound-iwaniec} and Lemma \ref{merten}, we have
    \begin{align*}
        \sum_{x^{1/2}<p<\xdelta} \varsigma(p)\log p &\le \sum_{x^{1/2}<p<\xdelta} \frac{(C(\theta)+o(1))x}{\phi(p)\log x}\rho(p)\log p\\
                                                    &= \frac{x}{\log x}\sum_{x^{1/2}<p<\xdelta} \frac{C(\theta)+o(1)}{\phi(p)}\rho(p)\log p\\
                                                    &= \frac{x}{\log x}\sum_{x^{1/2}<p<\xdelta} C(\theta)\frac{\rho(p)\log p }{p-1} + o\left(\frac{x\log\log x}{\log x}\right).\\
    \end{align*}
    It can be verified that the above inequality is true even when $f$ is an irreducible quadratic polynomial and we apply Theorem \ref{qbrun} instead of Theorem \ref{theta-bound-iwaniec}.
    By standard techniques to convert sums over primes into integrals, we have
    $$ \sum_{x^{1/2}<p<\xdelta} \varsigma(p)\log p \le (1+o(1))x\int_{1/2}^{\delta}C(\theta)~\mathrm{d}\theta,$$
    proving the lemma. 
\end{proof}

\subsection{The main bound}

It is easy to see that 
$$\log Q(x) = \sum_{p<x}(d\log p + \O(1))=dx+\O(x/\log x).$$
Define 
$$Q_{VL}(x)=\prod_{p\ge \xdelta}p^{\alpha_p(x)},$$
the part of $Q(x)$ consisting of primes at least $\xdelta$ (\textit{very large} primes).
Using Propositions \ref{small-prime-estimate}, \ref{medium-prime-estimate}, and \ref{large-prime-estimate}, we obtain
$$\log Q_{VL}(x) = \log \frac{Q(x)}{Q_S(x)Q_M(x)Q_L(x)}\ge \left(d-\frac 12 -\int_{1/2}^{\delta}C(\theta)~\mathrm{d}\theta+o(1)\right)x. $$

\begin{proposition}\label{main-bound}
    $\log Q_{VL}(x) \ge \left(d-\frac 12-\int_{1/2}^{\delta}C(\theta)~\mathrm{d}\theta+o(1)\right)x.$
\end{proposition}

\subsection{Bounding the integral}
The strategy will be to make $\delta$ as large as possible while keeping Proposition \ref{main-bound} non-trivial. Thanks to Theorem \ref{theta-bound-vanilla} and \ref{theta-bound-iwaniec}, we are able to bound the integral effortlessly. 
For $d \ge 2,$ 
\begin{align*}
    \int_{1/2}^{\delta}C(\theta)~\mathrm{d}\theta&=\int_{1/2}^{2/3}C(\theta)~\mathrm{d}\theta+\int_{2/3}^{\delta}C(\theta)~\mathrm{d}\theta\\
                                                 &< \int_{1/2}^{2/3}\frac{8}{6-7\theta}~\mathrm{d}\theta+\int_{2/3}^{\delta}\frac{2}{1-\theta}~\mathrm{d}\theta\\
                                                 &< -1.4788 -2\log(1-\delta).
\end{align*}
The case $d=1$ is a little special because we cannot make $\delta$ greater than $2/3.$ For $d=1,$
\begin{align*}
    \int_{1/2}^{\delta}C(\theta)~\mathrm{d}\theta < \int_{1/2}^{\delta}\frac{8}{6-7\theta}~\mathrm{d}\theta < 1.0472 - \frac 87\log(6 - 7\delta).
\end{align*}
\subsection{Choosing $\delta$} To preserve the linear lower bound in Proposition \ref{main-bound}, we want to have
$$d-\frac 12 \ge -1.4788 -2\log(1-\delta)$$
if $d\ge 2.$
This reduces to $\delta\le 1 - \exp\left(\frac{-d-0.9788}{2}\right).$ And for $d=1,$
$$1-\frac 12 \ge 1.0472 - \frac 87\log(6 - 7\delta)\implies \delta \le 0.62656.$$
However, we can do a lot better for $d=2,$ thanks to Theorem \ref{qbrun}. 
The following numerical computation, also performed in \cite{uwu}, shows that
\begin{align*}
    \int_{1/2}^{\delta}C(\theta)~\mathrm{d}\theta \le \int_{\frac 12}^{\frac{64}{97}}\frac{124}{91-89\theta}~\mathrm{d}\theta+\int_{\frac{64}{97}}^{\frac{32}{41}}\frac{120}{86-83\theta}~\mathrm{d}\theta +\int_{\frac{32}{41}}^{\delta}\frac{28}{19-18\theta}~\mathrm{d}\theta <\frac 32
\end{align*}
with $\delta=0.847.$
Thus, we set $\delta= 1 - \varepsilon(d)$ for the rest of the argument, where $\varepsilon(1)=0.3735,~\varepsilon(2)=0.153,$ and $\varepsilon(d)=\exp\left(\frac{-d-0.9788}{2}\right)$ for $d\ge 3$.

\subsection{Finishing the argument}
Define $L(x)=\lcm\{f(p)\mid p<x\}.$ Let $p$ be a prime such that $p\ge x^\delta.$ Note that the exponent of $p$ in $Q(x)$ is $\O( x^{1-\delta}).$ We know that $\log Q_{VL}(x) \gg x$. Therefore,
$$x\ll \log Q_{VL}(x)\ll x^{1-\delta}\sum_{\substack{p\ge\xdelta\\ p\mid Q(x)}} \log p.$$
Thus,
$$\log L(x) > \sum_{\substack{p\ge\xdelta\\ p\mid Q(x)}} \log p\gg \xdelta,$$
as desired.

\begin{remark}
    It is worth noting that the same method gives $\log \rad \lcm \{f(p) \mid p<x\}\gg x^{1-\varepsilon(d)}$, similar to that obtained by Sah in \cite{sah}.
\end{remark}

\section{Digression on the greatest prime divisor of $f(p)$}
The main ingredient in proving Theorem \ref{density-result} is Proposition \ref{main-bound}, which provides us a good handle on large primes dividing $Q(x).$ 
\begin{proof}[Proof of Theorem \ref{density-result}]
    By Proposition \ref{main-bound}, 
    $$\log Q_{VL}(x)=\sum_{q<x}\sum_{\substack{p>\xdelta\\ p\mid f(q)}}\log p\gg x.$$
    Set $\delta=1-\varepsilon(d).$ Let the number of primes $p$ less than $x$ such that $f(p)$ has a prime divisor greater than $x^{\delta}$ be $N(x).$ Note that if $p\mid Q(x),$ then $p< x^{d+1}$ for all large $x.$ Thus,
\begin{align*}
    N(x) \gg \sum_{q<x}\sum_{\substack{p>\xdelta\\ p\mid f(q)}} 1 \gg\sum_{q<x}\sum_{\substack{p>\xdelta\\ p\mid f(q)}} \frac{\log p}{\log x^{d+1}} \gg \frac{1}{\log x}\sum_{q<x}\sum_{\substack{p>\xdelta\\ p\mid f(q)}} \log p \gg \frac{x}{\log x},
    \end{align*}
   which completes the proof. 
\end{proof}

\begin{remark}
    It can be seen that the Elliott–Halberstam conjecture allows us to take $\varepsilon(d)$ to be any positive constant. For completeness, a formulation of the Elliott-Halberstam conjecture is as follows:

\vspace{0.2cm}

\noindent\textbf{Elliott-Halberstam Conjecture.}\textit{
    Define the error function 
    $$E(x;q)=\max_{\gcd(a,q)=1} \left|\pi(x;q,a)-\frac{\pi(x)}{\varphi(q)}\right|,$$
    where the $\max$ is taken over all $a$ relatively prime to $q.$ For every $\theta<1$ and $A>0,$ we have
    $$\sum_{1\le q\le x^{\theta}} E(x;q) \ll_{\theta,A} \frac{x}{\log^A x}.$$
}
\end{remark}

We end the article with the following question for readers.
\begin{question}\label{main-conjecture}
    Let $f$ be an irreducible integer polynomial. Is it true that 
    $\log \lcm \{f(p)\mid p<x\} \gg x?$
\end{question}

\bibliographystyle{amsplain}

\begin{thebibliography}{99}

\bibitem{sanna}
D.~Bazzanella and C.~Sanna, 
\newblock Least common multiple of polynomial sequences,
\newblock {\em Rend. Semin. Mat. Univ. Politec. Torino} {\bf 78}(1) (2020) 21--25.

\bibitem{cill}
J.~Cilleruelo,
\newblock The least common multiple of a quadratic sequence.
\newblock {\em Compos. Math.} {\bf 147}(4) (2011) 1129--1150.

\bibitem{fouvry}
E. Fouvry,
\newblock Théorème de Brun-Titchmarsh; application au théorème der Fermat,
\newblock {\em Invent. Math.} {\bf 79} (1985) 383-408. 

\bibitem{goldfeld}
M. Goldfeld, 
\newblock On the number of primes $p$ for which $p+a$ has a large prime factor,
\newblock {\em Mathematika} {\bf 16}(1) (1969) 23--27.

\bibitem{prime-detecting-sieves}
G.~Harman,
\newblock \textit{Prime-Detecting Sieves} (Princeton University Press, 2007).

\bibitem{hong:lower-bound}
S.~Hong, Y.~Luo, G.~Qian and C.~Wang,
\newblock Uniform lower bound for the least common multiple of a polynomial
  sequence,
  \newblock {\em C. R. Math. Acad. Sci. Paris} {\bf 351}(21-22) (2013) 781--785.

\bibitem{hong}
S.~Hong, G.~Qian and Q.~Tan,
\newblock The least common multiple of a sequence of products of linear
  polynomials,
  \newblock {\em Acta Math. Hungar.} {\bf 135}(1-2) (2012) 160--167.

\bibitem{iwaniec}
H. Iwaniec,
\newblock On the Brun-Titchmarsh theorem,
\newblock {\em J. Math. Soc. Japan} {\bf 34}(1) (1982) 95--123. 

\bibitem{luca:p}
F. Luca, R. Menares and A. Pizarro-Madariaga, On shifted primes with large prime factors and their products, {\em Bull. Belg. Math. Soc. Simon Stevin} {\bf 22} (2015) 39--47.

\bibitem{maynard}
J.~Maynard and Z.~Rudnick,
\newblock A lower bound on the least common multiple of polynomial sequences,
\newblock {\em Riv. Mat. Univ. Parma} {\bf 12}(1) (2021) 143--15.

\bibitem{brun}
H.~L.~Montgomery and R.~C.~Vaughan,
\newblock The Large Sieve. 
\newblock {\em Mathematika} {\bf 20}(02) (1973) 119. 

\bibitem{sah}
A. Sah, 
\newblock An improved bound on the least common multiple of polynomial sequences,
\newblock {\em J. Théor. Nombres Bordeaux} {\bf 32}(3) (2020) 891--899. 

\bibitem{serre}
J.-P. Serre,
\newblock \textit{Lectures on $N_X(p)$}
\newblock (CRC Press Book, Research Notes in Mathematics, 2011).

\bibitem{uwu}
J. Wu and P. Xi,
\newblock Quadratic polynomials at prime arguments,
\newblock \textit{Math. Z.} {\bf 285} (2017) 631--646. 

\bibitem{qbrun}
J. Wu and P. Xi,
\newblock Arithmetic exponent pairs for algebraic trace functions,
\newblock \textit{to appear in Algebra Number Theory}.

\end{thebibliography}

\end{document}